\documentclass[11pt]{amsart}

\usepackage{amssymb}
\usepackage{amsmath}
\usepackage[all,cmtip]{xy}
\usepackage{amsfonts}
\usepackage{mathrsfs}
\usepackage{latexsym}
\usepackage{graphicx}
\usepackage{amscd,amssymb,amsmath,amsbsy,amsthm}
\usepackage[colorlinks,plainpages,backref,urlcolor=blue]{hyperref}
\allowdisplaybreaks

\usepackage{tikz}
\usetikzlibrary{arrows,calc}
\tikzset{
>=stealth',
help lines/.style={dashed, thick},
axis/.style={<->},
important line/.style={thick},
connection/.style={thick, dotted},
}

\newcommand{\nc}{\newcommand}
\nc{\rnc}{\renewcommand}
\nc{\bb}[1]{{\mathbb #1}}
\nc{\bbA}{\bb{A}}\nc{\bbB}{\bb{B}}\nc{\bbC}{\bb{C}}\nc{\bbD}{\bb{D}}
\nc{\bbE}{\bb{E}}\nc{\bbF}{\bb{F}}\nc{\bbG}{\bb{G}}\nc{\bbH}{\bb{H}}
\nc{\bbI}{\bb{I}}\nc{\bbJ}{\bb{J}}\nc{\bbK}{\bb{K}}\nc{\bbL}{\bb{L}}
\nc{\bbM}{\bb{M}}\nc{\bbN}{\bb{N}}\nc{\bbO}{\bb{O}}\nc{\bbP}{\bb{P}}
\nc{\bbQ}{\bb{Q}}\nc{\bbR}{\bb{R}}\nc{\bbS}{\bb{S}}\nc{\bbT}{\bb{T}}
\nc{\bbU}{\bb{U}}\nc{\bbV}{\bb{V}}\nc{\bbW}{\bb{W}}\nc{\bbX}{\bb{X}}
\nc{\bbY}{\bb{Y}}\nc{\bbZ}{\bb{Z}}
\nc{\mbf}[1]{{\mathbf #1}}
\nc{\bfA}{\mbf{A}}\nc{\bfB}{\mbf{B}}\nc{\bfC}{\mbf{C}}\nc{\bfD}{\mbf{D}}
\nc{\bfE}{\mbf{E}}\nc{\bfF}{\mbf{F}}\nc{\bfG}{\mbf{G}}\nc{\bfH}{\mbf{H}}
\nc{\bfI}{\mbf{I}}\nc{\bfJ}{\mbf{J}}\nc{\bfK}{\mbf{K}}\nc{\bfL}{\mbf{L}}
\nc{\bfM}{\mbf{M}}\nc{\bfN}{\mbf{N}}\nc{\bfO}{\mbf{O}}\nc{\bfP}{\mbf{P}}
\nc{\bfQ}{\mbf{Q}}\nc{\bfR}{\mbf{R}}\nc{\bfS}{\mbf{S}}\nc{\bfT}{\mbf{T}}
\nc{\bfU}{\mbf{U}}\nc{\bfV}{\mbf{V}}\nc{\bfW}{\mbf{W}}\nc{\bfX}{\mbf{X}}
\nc{\bfY}{\mbf{Y}}\nc{\bfZ}{\mbf{Z}}
\nc{\bfa}{\mbf{a}}\nc{\bfb}{\mbf{b}}\nc{\bfc}{\mbf{c}}\nc{\bfd}{\mbf{d}}
\nc{\bfe}{\mbf{e}}\nc{\bff}{\mbf{f}}\nc{\bfg}{\mbf{g}}\nc{\bfh}{\mbf{h}}
\nc{\bfi}{\mbf{i}}\nc{\bfj}{\mbf{j}}\nc{\bfk}{\mbf{k}}\nc{\bfl}{\mbf{l}}
\nc{\bfm}{\mbf{m}}\nc{\bfn}{\mbf{n}}\nc{\bfo}{\mbf{o}}\nc{\bfp}{\mbf{p}}
\nc{\bfq}{\mbf{q}}\nc{\bfr}{\mbf{r}}\nc{\bfs}{\mbf{s}}\nc{\bft}{\mbf{t}}
\nc{\bfu}{\mbf{u}}\nc{\bfv}{\mbf{v}}\nc{\bfw}{\mbf{w}}\nc{\bfx}{\mbf{x}}
\nc{\bfy}{\mbf{y}}\nc{\bfz}{\mbf{z}}

\nc{\mcal}[1]{{\mathcal #1}}
\nc{\calA}{\mcal{A}}\nc{\calB}{\mcal{B}}\nc{\calC}{\mcal{C}}\nc{\calD}{\mcal{D}}
\nc{\calE}{\mcal{E}} \nc{\calF}{\mcal{F}}\nc{\calG}{\mcal{G}}\nc{\calH}{\mcal{H}}
\nc{\calI}{\mcal{I}}\nc{\calJ}{\mcal{J}}\nc{\calK}{\mcal{K}}\nc{\calL}{\mcal{L}}
\nc{\calM}{\mcal{M}}\nc{\calN}{\mcal{N}}\nc{\calO}{\mcal{O}}\nc{\calP}{\mcal{P}}
\nc{\calQ}{\mcal{Q}}\nc{\calR}{\mcal{R}}\nc{\calS}{\mcal{S}}\nc{\calT}{\mcal{T}}
\nc{\calU}{\mcal{U}}\nc{\calV}{\mcal{V}}\nc{\calW}{\mcal{W}}\nc{\calX}{\mcal{X}}
\nc{\calY}{\mcal{Y}}\nc{\calZ}{\mcal{Z}}
\nc{\fA}{\frak{A}}\nc{\fB}{\frak{B}}\nc{\fC}{\frak{C}} \nc{\fD}{\frak{D}}
\nc{\fE}{\frak{E}}\nc{\fF}{\frak{F}}\nc{\fG}{\frak{G}}\nc{\fH}{\frak{H}}
\nc{\fI}{\frak{I}}\nc{\fJ}{\frak{J}}\nc{\fK}{\frak{K}}\nc{\fL}{\frak{L}}
\nc{\fM}{\frak{M}}\nc{\fN}{\frak{N}}\nc{\fO}{\frak{O}}\nc{\fP}{\frak{P}}
\nc{\fQ}{\frak{Q}}\nc{\fR}{\frak{R}}\nc{\fS}{\frak{S}}\nc{\fT}{\frak{T}}
\nc{\fU}{\frak{U}}\nc{\fV}{\frak{V}}\nc{\fW}{\frak{W}}\nc{\fX}{\frak{X}}
\nc{\fY}{\frak{Y}}\nc{\fZ}{\frak{Z}}
\nc{\fa}{\frak{a}}\nc{\fb}{\frak{b}}\nc{\fc}{\frak{c}} \nc{\fd}{\frak{d}}
\nc{\fe}{\frak{e}}\nc{\fFf}{\frak{f}}\nc{\fg}{\frak{g}}\nc{\fh}{\frak{h}}
\nc{\fri}{\frak{i}}\nc{\fj}{\frak{j}}\nc{\fk}{\frak{k}}\nc{\fl}{\frak{l}}
\nc{\fm}{\frak{m}}\nc{\fn}{\frak{n}}\nc{\fo}{\frak{o}}\nc{\fp}{\frak{p}}
\nc{\fq}{\frak{q}}\nc{\fr}{\frak{r}}\nc{\fs}{\frak{s}}\nc{\ft}{\frak{t}}
\nc{\fu}{\frak{u}}\nc{\fv}{\frak{v}}\nc{\fw}{\frak{w}}\nc{\fx}{\frak{x}}
\nc{\fy}{\frak{y}}\nc{\fz}{\frak{z}}

\newtheorem{theorem}{Theorem}[section]
\newtheorem{lemma}[theorem]{Lemma}
\newtheorem{corollary}[theorem]{Corollary}
\newtheorem{prop}[theorem]{Proposition}

\theoremstyle{definition}
\newtheorem{definition}[theorem]{Definition}
\newtheorem{example}[theorem]{Example}
\newtheorem{remark}[theorem]{Remark}

\newtheorem{thm}{Theorem}

 \DeclareMathOperator{\id}{id}
 
 \DeclareMathOperator{\supp}{supp}
 
\DeclareMathOperator{\Hom}{{Hom}}

 \DeclareMathOperator{\End}{End}

\DeclareMathOperator{\res}{res}

\DeclareMathOperator{\HH}{HH}

\newcommand{\cQ}{\mathcal {Q}}

\newcommand{\fraks}{\mathfrak{s}}
\newcommand{\al}{\alpha}
\newcommand{\la}{\lambda}

\newcommand{\ka}{\kappa}

\newcommand{\de}{\delta}

\newcommand{\frakt}{\mathfrak{t}}
\newcommand{\be}{\beta}

\DeclareMathOperator{\pr}{pr}

\newcommand{\aff}{{\operatorname{a}}}

\newcommand{\dia}{\diamond}
\newcommand{\frakX}{\mathfrak{X}}
\newcommand{\frakY}{\mathfrak{Y}}

\DeclareMathOperator{\loc}{{loc}}
\newcommand{\hatot}{\hat\otimes}
\newcommand{\lag}{\langle}
\newcommand{\rag}{\rangle}
\def\Hom{\operatorname{Hom}}
\def\res{\operatorname{res}}

\newcount\cols
{\catcode`,=\active\catcode`|=\active
 \gdef\Young(#1){\hbox{$\vcenter
 {\mathcode`,="8000\mathcode`|="8000
  \def,{\global\advance\cols by 1 &}%
  \def|{\cr
        \multispan{\the\cols}\hrulefill\cr
        &\global\cols=2 }%
  \offinterlineskip\everycr{}\tabskip=0pt
  \dimen0=\ht\strutbox \advance\dimen0 by \dp\strutbox
  \halign
   {\vrule height \ht\strutbox depth \dp\strutbox##
    &&\hbox to \dimen0{\hss$##$\hss}\vrule\cr
    \noalign{\hrule}&\global\cols=2 #1\crcr
    \multispan{\the\cols}\hrulefill\cr%
   }
 }$}}
}

\author[R.~Xiong]{Rui~Xiong}
\address{Department of Mathematics and Statistics, University of Ottawa, 150 Louis-Pasteur, Ottawa, ON, K1N 6N5, Canada}
\email{rxiong@uottawa.ca}

\author[K.~Zainoulline]{Kirill~Zainoulline}
\address{Department of Mathematics and Statistics, University of Ottawa, 150 Louis-Pasteur, Ottawa, ON, K1N 6N5, Canada}
\email{kirill@uottawa.ca}

\author[C.~Zhong]{Changlong~Zhong}
\address{State University of New York at Albany, 1400 Washington Avenue, Albany, NY 12222, USA}
\email{czhong@albany.edu}

\title{On the formal Peterson subalgebra and its dual}

\begin{document}

\begin{abstract} 
In the present notes, we study a generalization of the Peterson subalgebra to an oriented (generalized) cohomology theory which we call the formal Peterson subalgebra. Observe that by recent results of Zhong the dual of the formal Peterson algebra provides an algebraic model for the  oriented cohomology of the affine Grassmannian.

Our first result shows that the centre of the formal affine Demazure algebra generates the formal Peterson subalgebra. Our second observation is motivated by the Peterson conjecture. We show that a certain localization of the formal Peterson subalgebra for the extended Dynkin diagram of type $\hat A_1$ provides an algebraic model for `quantum' oriented  cohomology of the projective line. Our last result can be viewed as an extension of the previous results on Hopf algebroids of structure algebras of moment graphs to the case of affine root systems.
We prove that the dual of the formal Peterson subalgebra (an oriented cohomology of the affine Grassmannian) is the $0$th Hochshild homology of the formal affine Demazure algebra.
\end{abstract}

\maketitle

%%%%%%%%%%%%%%%%%%%%%%%%%%%%%%%
%%%%%%%%%%%%%%%%%%%%%%%%%%%%%%%
%%%%%%%%%%%%%%%%%%%%%%%%%%%%%%%

\section{Introduction}
Equivariant cohomology of an affine Grassmannian has been a topic of intensive investigations for decades. For the small torus action it can be identified with a certain commutative subalgebra of the associated nil-Hecke algebra of a Kac-Moody root system called the Peterson subalgebra \cite{P97}. One of its remarkable properties says that after taking localization it becomes isomorphic to the (small) quantum cohomology of the respective finite part (flag variety) \cite{P97,LS10}. A parallel isomorphism for the $K$-theory was conjectured and discussed in \cite{LSS10,LLMS18} and is known as the Peterson Conjecture. This conjecture was recently proven by Kato in \cite{K18} using a language of semi-infinite flag varieties.

In the present notes, we study a generalization of the Peterson subalgebra to an oriented (generalized) cohomology theory $h(\text{-})$, e.g. algebraic cobordism $\Omega(\text{-})$ of Levine-Morel. Such a cohomology theory was first introduced and studied in \cite{LM07}, and extended to the torus-equivariant setup in \cite{Kr12,HM13} for arbitrary smooth varieties. As for flag varieties associated to root systems, it can be described using the Kostant-Kumar localization approach (for finite root systems see \cite{CZZ1, CZZ2, CZZ3}, and for Kac-Moody see \cite{CZZ4}). The respective generalization of the nil Hecke algebra is called the formal affine Demazure algebra (FADA). The generalization of the Peterson algebra introduced recently in \cite{Z23} which we call a formal Peterson subalgebra is then the centralizer of the equivariant coefficient ring in the small torus FADA.

To state our first result,  let $R=h(pt)$ denote the coefficient ring of the oriented theory $h$, let $S=h_T(pt)$ denote the respective small torus $T$ equivariant coefficient ring, let $\bfD_{W_\aff}$ denote the small torus FADA, and let $\bfD_{Q^\vee}$ denote the formal Peterson subalgebra as constructed in \cite{Z23}. We then obtain the following important property of the centre of FADA:

\begin{thm}[cf. Theorem~\ref{thm:main1}]
If $\mathbb{Q}\subseteq R$, then the centre $Z(\bfD_{W_\aff})$ of the small torus FADA generates the formal Peterson subalgebra $\bfD_{Q^\vee}$ as an $S$-module. Moreover, the centre $Z(\bfD_{W_\aff})$ generates $\bfD_{W_\aff}$ as a $\bfD_W$-module, where $\bfD_W$ stands for the FADA associated to the finite part of the Kac-Moody root system. \end{thm}

Our next result can be viewed as an extension of the Peterson conjecture.
We first introduce certain localization $\bfD_{Q^\vee, \loc}$ of the formal Peterson subalgebra $\bfD_{Q^\vee}$ with respect to an affine root system of type $\hat A_1$.  We then show

\begin{thm}[cf. Theorem~\ref{thm:main2}]
$\bfD_{Q^\vee, \loc}\cong S[\frakt,{\frakt}^{-1}][\fraks]/(\fraks^2=x_{-1}\fraks\frakt+\mu \frakt)$,
where $x_{-1}$ is a certain characteristic class in $h$ and 
$\mu$ is an element depending on $x_{-1}$.
\end{thm}

Observe that for cohomology and $K$-theory, this localization computes quantum cohomology and quantum $K$-theory of $\bbP^1$ respectively. Hence, $\bfD_{Q^\vee, \loc}$ can be viewed as the `quantum' oriented cohomology of the projective line $\bbP^1$.

As for our last result, observe that the $S$-linear dual $\bfD_{Q^\vee}^*$ of the formal Peterson subalgebra is a natural model for the (small torus) equivariant oriented cohomology of the affine Grassmannian \cite{Z23}. We obtain the following `Kac-Moody' version of results of~\cite{LXZ23}: 

\begin{thm}[cf. Theorem~ \ref{thm:main3}]
The $S$-linear dual $\bfD_{Q^\vee}^*$ of the formal Peterson subalgebra is isomorphic to the $0$-th Hochschild homology of the dual $\bfD_{W_\aff}^*$ of the small torus FADA.
\end{thm}

Here the dual $\bfD_{W_\aff}^*$ can be interpreted as a model for the $T$-equivariant oriented cohomology of the respective affine flag variety. Therefore, it has two commuting actions by the equivariant coefficient ring $S$. Following the ideas of \cite{LXZ23} one defines its 0-th Hochschild homology as the quotient obtained by merging these two $S$-module structures. 

To prove this result we introduce a special filtration on the dual $\bfD_{Q^\vee}^*$ (to reduce it to finite cases). This approach seems to be new even for cohomology and the $K$-theory.

\medskip

The paper is organized as follows: Section~\ref{sec1} revisits the definition of the formal Peterson subalgebra $\bfD_{Q^\vee}$ from \cite{Z23}. In Section 2 we establish some basic properties of $\bfD_{Q^\vee}$ and study the action of $\bfD_{W_\aff}$ on it. Section~3 is dedicated to the study of $\bfD_{Q^\vee}$ and $\bfD_{W_\aff}$, culminating in the proof of Theorem~\ref{thm:main1}. Section~4 focuses on the example of type $\hat A^1$ and on the proof of Theorem~\ref{thm:main2}. Finally, in Section~5 we investigate the dual of the formal Peterson subalgebra, and we prove Theorem~\ref{thm:main3}.

\subsection*{Acknowledgements} Part of this research was conducted when C.~Z.~ was visiting the University of Ottawa, and the Max Planck Institute of Mathematics.  He  would like to thank these institutes for their hospitality and support. R.~X. and K.~Z. were partially supported by the NSERC Discovery grant RGPIN-2022-03060, Canada.

%%%%%%%%%%%%%%%%%%%%%%%%%%%%%%
%%%%%%%%%%%%%%%%%%%%%%%%%%%%%%
%%%%%%%%%%%%%%%%%%%%%%%%%%%%%%

\section{The formal Peterson subalgebra}\label{sec1}

In this section we recall the definition of a small torus formal affine Demazure algebra (FADA) and the formal Peterson algebra
following \cite{CZZ4,Z23}.

Given an oriented algebraic cohomology theory $h(\text{-})$ in the sense of Levine-Morel (see  \cite{LM07}) there is an associated formal group law $F$ over a commutative ring $R$ with characteristic 0. Here $R=h(pt)$ is the coefficient ring, and $F$ is defined from the Quillen formula for the characteristic class of a tensor product of line bundles. For example, for connective $K$-theory (see e.g. \cite{RZ23}) we have $F_\be(x,y)=x+y-\be xy$ over the polynomial ring $R=\bbZ[\be]$. Specializing to $\be=1$ (resp. $\be=0$) one obtains the usual $K$-theory (resp. cohomology). In these notes, by usual cohomology we always mean its algebraic part: the Chow ring (modulo rational equivalence) with rational coefficients.

Given a lattice $\Lambda$ (free abelian group of finite rank) and a formal group law $F$, consider the associated formal group algebra $S$ of \cite{CPZ13} that is the quotient of the power series ring \[S=R[[\Lambda]]_F=R[[x_\la \mid \la\in \Lambda]]/\calJ_F,\] where $\calJ_F$ is the closure of the ideal of relations ($x_0$, $x_{\la_1+\la_2}-F(x_{\la_1}, x_{\la_2})$) for all $\la_1,\la_2\in \Lambda$. In the case $F=F_\be$ we will take $S$ to be the quotient $R[\Lambda]_F=R[x_\la\mid \la\in \Lambda]/\langle x_0,x_{\la_1+\la_2}-F_\be(x_{\la_1}, x_{\la_2})\rangle$ of the polynomial ring.

Consider a finite irreducible root system. Let $\Phi$ denote the set of roots, and $I=\{\al_1,\ldots,\al_n\}$ denote the set of simple roots. Let $Q$ be the root lattice, and $Q^\vee$ be the coroot lattice.  Let $W$ denote the Weyl group, generated by simple reflections $s_{\al_i}$, $\al_i\in I$. Let $\theta\in \Phi$ be the highest root. Consider an affine root system corresponding to the extended Dynkin diagram for $\Phi$. Let $\al_0=-\theta+\delta\in Q\oplus \bbZ\delta$ denote the  extra simple root (here $\delta$ is the so called null root) so that 
$s_{\al_0}= t_{\theta^\vee} s_\theta$ is an extra generator of the respective affine Weyl group $W_\aff= Q^\vee \rtimes W$.  
Recall that the latter is generated by reflections $s_{\al+k\delta}= t_{-k\al^\vee} s_\al$, where $s_\al\in W$ is a reflection and $t_\la$ is a translation. The affine Weyl group $W_\aff$ acts on $Q$ via $W$ that is $t_\la w(\mu)=w(\mu)$, $\mu\in Q$, $w\in W$, $\la\in Q^\vee$.
Therefore, it also acts on the formal group algebra $S=R[[Q]]_F$.

Suppose $x_\alpha$ is regular for each $\alpha$ (see \cite[Lemma~2.2 and Definition 4.4]{CZZ1}). In particular, this holds if $2$ is not a zero-divisor in $R$.
 Consider the localization $\cQ=S[\tfrac{1}{x_\al}\mid \al\in \Phi]$ of $S$.  Define the twisted group algebra $\cQ_{W_\aff}=\cQ\otimes_RR[W_\aff]$. It is a free left $\cQ$-module with basis $\eta_u$, $u\in W_\aff$. It is also an $R$-algebra with the product given by  \[c\eta_u\cdot c'\eta_{u'}=cu(c')\eta_{uu'}, ~c,c'\in \cQ, ~u,u'\in W_\aff.\]

For each $\al\in \Phi$, define $\ka_\al=\frac{1}{x_\al}+\frac{1}{x_{-\al}}$ (which is an element of $S$), $X_\al=\frac{1}{x_\al}(1-\eta_{s_\al})$ (the Demazure element) and $Y_\al=\ka_\al-X_\al$ (the push-pull element). Also define $X_{\al_0}=\frac{1}{x_{-\theta}}(1-\eta_{s_0})$ and $Y_{\al_0}=\ka_\theta-X_{\al_0}$. All these elements satisfy the quadratic relations (e.g. $X_\al^2=\ka_\al X_\al$) and the twisted braid relations (see e.g. \cite{CZZ4}). We denote by $X_{I_u}$, $Y_{I_u}$ products corresponding to a reduced sequence $I_u$ of $u\in W_\aff$. For simplicity of notations, we will omit $\alpha$ or $s$ in the indices, i.e., we will write
$x_i=x_{\al_i}$, $s_i=s_{\al_i}$, $\eta_i=\eta_{s_{i}}$, $X_i=X_{\al_i}$, and  $Y_i=Y_{\al_i}$.

Similarly, define the twisted group algebra $\cQ_{Q^\vee}=\cQ\otimes_RR[Q^\vee]$. It is a free $\cQ$-module with basis $\eta_{t_\la}$, $\la\in Q^\vee$.  Observe that $\cQ_{Q^\vee}$ is commutative since $t_\la(c)=c$, $c\in \cQ$. Consider two homomorphisms of left $\cQ$-modules 
\begin{align*}
\pr\colon&\cQ_{W_\aff}\to \cQ_{Q^\vee}, ~c\eta_{t_\la w}\mapsto c\eta_{t_\la}, ~c\in \cQ, w\in W,\\
\imath\colon&\cQ_{Q^\vee}\to \cQ_{W_\aff}, ~c\eta_{t_\la}\mapsto c\eta_{t_\la}.
\end{align*}
By definition, $\imath$ is a section of $\pr$, and it is a ring homomorphism. Set $\psi=\imath\circ \pr$, so $\psi|_{\imath(\cQ_{Q^\vee})}=\id$. 
Define elements $Z_\al=\frac{1}{x_{-\al}}(1-\eta_{t_{\al^\vee}})$, $\al\in \Phi$. 

\begin{lemma} 
We have  $\psi(zX_i)=0$  for any $z\in \cQ_{W_\aff}$, $\al_i\in I$, and $\psi(X_0)=Z_\theta$. 
\end{lemma}

\begin{proof}
For $z=c\eta_u$, $c\in \cQ$, $u\in W_\aff$, and $\al_i\in I$ we have
\[
\pr (c\eta_u X_i)=\pr\big(c\eta_u\tfrac{1}{x_{\al_i}}(1-\eta_{s_i})\big)=\pr\big(\tfrac{c}{u(x_{\al_i})} (\eta_u-\eta_{us_i})\big)=0.
\]
As for the $X_0$, we have
\[
\psi(X_0)=\imath\circ \pr(\tfrac{1}{x_{-\theta}}(1-\eta_{s_0}))=\tfrac{1}{x_{-\theta}}\imath\circ \pr((1-\eta_{t_{\theta^\vee}s_\theta}))=\tfrac{1}{x_{-\theta}}(1-\eta_{t_{\theta^\vee}})=Z_{\theta}. \qedhere
\]
\end{proof}

Denote $I_\aff=\{\al_0,\ldots,\al_n\}$. Following~\cite{CZZ4} define the small torus FADA $\bfD_{W_\aff}$ to be the subring of $\cQ_{W_\aff}$ generated by $S$ and the elements $X_i$, $\al_i\in I_\aff$. Set $\bfD_{W_\aff/W}=\pr(\bfD_{W_\aff})$ to be the image in $\cQ_{Q^\vee}$. We called it the relative FADA. Assuming $\mathbb{Q}\subset R$ and using the small torus GKM  descripition it is proven in \cite[Lemma 5.1]{Z23} that the map $\imath$ induces a map $\bfD_{W_\aff/W}\to \bfD_{W_\aff}$. We then introduce a formal Peterson subalgebra to be the image $\bfD_{Q^\vee}=\imath(\bfD_{W_\aff/W})$ of the relative FADA.
One of the results in \cite[Theorem 5.7]{Z23} says that 

\begin{prop}\label{prop:hopf}  Suppose $R$ has characteristic 0, then the formal Peterson subalgebra $\bfD_{Q^\vee}$ is a Hopf subalgebra in $\cQ_{Q^\vee}$. Moreover,  $\bfD_{Q^\vee}$ coincides with the centralizer $C_{\bfD_{W_\aff}}(S)$ of the formal group algebra $S$ in the FADA $\bfD_{W_\aff}$. 
\end{prop}

%%%%%%%%%%%%%%%%%%%%%%%%%%%%%%%%

\section{Properties of the FADA and the Peterson subalgebra}

In the present section we establish several properties of the FADA and the Peterson subalgebra.
For  $K$-theory some of these properties were proven in \cite{K18} with different arguments.  
We start with the following version of the projection formula:

\begin{lemma}\label{lem:pr}
For any $z$, $z'\in \cQ_{W_\aff}$ and  $\xi\in \cQ_{Q^\vee}$ we have in $\cQ_{Q^\vee}$:
\begin{itemize}
\item[(i)] $\pr(\imath(\xi)z)=\xi\pr(z)$, and
\item[(ii)]  $\pr(z\sigma z')=\pr(z)\pr(\sigma z')$, where  $\sigma=\sum_{w\in W}\eta_w$. 
\end{itemize}
\end{lemma}

Observe that in the $K$-theory case the second property played a key role in \cite[Theorem~1.7]{K18}. 

\begin{proof}
(i) Let $\xi=c_1\eta_{t_{\la_1}}$ and  $z=c_2\eta_{t_{\la_2}w}$, where $w\in W$, $c_i\in \cQ$, $\la_i\in Q^\vee$. Then we obtain
\[
\pr(\imath(\xi)z)=\pr(c_1\eta_{t_{\la_1}}c_2\eta_{t_{\la_2}w})=\pr(c_1t_{\la_1}(c_2)\eta_{t_{\la_1+\la_2}w})=c_1t_{\la_1}(c_2)\eta_{t_{\la_1+\la_2}}=c_1\eta_{t_{\la_1}}c_2\eta_{t_{\la_2}}=\xi \pr (z). 
\]

\noindent
(ii) Let $z=c\eta_{t_{\la}v}$ and $z'=c'\eta_{t_{\la'}v'}$, $v,v'\in W$. Then we get 
\[
\pr(z\sigma z')=\pr\big(c\eta_{t_{\la}v}\sum_{w\in W}\eta_w c'\eta_{t_{\la'}v'}\big)=\pr\big(\sum_{w\in W}ct_\la vw(c')\eta_{t_\la vwt_{\la'}v'}\big).
\]
Since $t_\la vw t_{\la'} v'=t_\la (vwt_{\la'}(vw)^{-1}) vw v'=t_\la t_{vw(\la')}vwv'$, reindexing the sum by $w'=vw$ we obtain
\[
\pr(z\sigma z') =\pr(\sum_{w'\in W}ct_\la w'(c')\eta_{t_{\la}t_{w'(\la')}w'v'}) =\sum_{w'\in W}ct_\la w'(c')\eta_{t_{\la}t_{w'(\la')}}.
\]
On the other side  $\pr(z)=c\eta_{t_\la}$ and
\[
\pr(\sigma z')=\pr\big(\sum_w \eta_w c'\eta_{t_{\la'}v'}\big)=\pr\big(\sum_w w(c') \eta_{wt_{\la'}v'}\big)=\pr\big(\sum_w w(c') \eta_{t_{w(\la')}wv'}\big)=\sum_w w(c') \eta_{t_{w(\la')}}.
\]
The result then follows.
\end{proof}

We now extend the Hecke action on the Peterson algebra for the $K$-theory introduced in \cite[\S2]{K18} to the action on the formal Peterson algebra:

We define an action of $\cQ_{W_\aff}$ on $\cQ_{Q^\vee}$ by 
\[
z\dia \xi=\pr(z\imath(\xi)), ~z\in \cQ_{W_\aff}, \xi\in \cQ_{Q^\vee}. 
\]
More explicitly, we have
\begin{equation}\label{eq:dia}
c\eta_{t_\la w}\dia c'\eta_{t_{\la'}}=cw(c')\eta_{t_{\la+w(\la')}}, ~c,c'\in \cQ, w\in W.
\end{equation}
Direct computation shows that  $\dia$ is an action.

For $w\in W$, $\xi\in \cQ_{Q^\vee}$ and $\al\in\Phi$ define
\[
w(\xi)=\eta_w\dia \xi, \;\text{and}\; \Delta_\al(\xi)=X_\al\dia \xi=\tfrac{1}{x_\al}(\xi-s_\al( \xi)). 
\]

\begin{lemma}\label{lem:pr2}
For any $z,z'\in  \cQ_{W_\aff}$ and $\xi\in \cQ_{Q^\vee}$ we have
\begin{itemize}
\item[(i)] $\pr(z z')=z \dia \pr(z')$, in particular, $\pr(X_iz)=X_i\dia \pr(z)=\Delta_i( \pr(z)) $, $\alpha_i\in I$,
\item[(ii)]  $X_0\dia \xi=\Delta_{-\theta}(\xi)+Z_\theta s_{\theta}(\xi)$.
\end{itemize}
\end{lemma}

\begin{proof}
(i) Let $z=c\eta_{t_{\la}w}$ and $z'=c'\eta_{t_{\la'}w'}$. Then we obtain
\[
\pr(zz')=\pr(c\eta_{t_{\la}w}c'\eta_{t_{\la'}w'})=\pr(ct_{\la}w(c')\eta_{t_{\la+w(\la')}ww'})=ct_{\la}w(c')\eta_{t_{\la+w(\la')}}=c\eta_{t_{\la}w}\dia c'\eta_{t_{\la'}}.
\]

\noindent
(ii) For $\xi=c\eta_{t_\la}$ we get 
\begin{align*}
X_0\dia (c\eta_{t_\la})&=\tfrac{c}{x_{-\theta}}\eta_{t_\la}-\tfrac{s_0(c)}{x_{-\theta}}\eta_{t_{s_\theta(\la)}t_{\theta^\vee}}\\
&=\tfrac{c}{x_{-\theta}}\eta_{t_\la}-\tfrac{s_\theta(c)}{x_{-\theta}}\eta_{t_{s_\theta(\la)}t_{\theta^\vee}}-\tfrac{s_\theta(c)}{x_{-\theta}}\eta_{t_{s_\theta(\la)}}+\tfrac{s_\theta(c)}{x_{-\theta}}\eta_{t_{s_\theta(\la)}}\\
&=\Delta_{-\theta}(c\eta_{t_\la})+Z_{\theta}s_\theta(c\eta_{t_\la}). \qedhere
\end{align*}
\end{proof}

\begin{lemma}The $\dia$-action of $\cQ_{W_\aff}$ on $\cQ_{Q^\vee}$ induces an action of $\bfD_{W_\aff}$ on $\bfD_{W_\aff/W}$.
\end{lemma}

\begin{proof} Let $z,z'\in \bfD_{W_\aff}$, and let $\xi=\pr(z')$. Then we have 
\[z\dia \xi=z\dia \pr (z')=\pr (zz')\in \pr(\bfD_{W_\aff})=\bfD_{W_\aff/W}.\qedhere \]
\end{proof}

Identifying the formal Peterson algebra $\bfD_{Q^\vee}$ (resp. $\imath(\cQ_{Q^\vee})$) with $\bfD_{W_\aff/W}$ (resp. $\cQ_{Q^\vee}$) via the ring homomorphism $\imath$ we obtain an action of $\bfD_{W_\aff}$  on $\bfD_{Q^\vee}$ and an action of $\cQ_{W_\aff}$ on $\imath(\cQ_{Q^\vee})$.  From this point on, we write $\xi$ as both an element in $\bfD_{W_\aff/W}$ (resp. $\cQ_{Q^\vee}$) and in $\bfD_{Q^\vee}=\imath(\bfD_{W_\aff/W})$ (resp. $\imath(\cQ_{Q^\vee})$). If we consider a product $\xi_1\xi_2$ with $\xi_i\in \cQ_{Q^\vee}$, we may assume it is in $\cQ_{W_\aff}$. However, for the product $z\xi$ with $z\in \cQ_{W_\aff}$ and $\xi\in \cQ_{Q^\vee}$, we view $\xi$ as an element  in $\cQ_{W_\aff}$ via the map $\imath$. Following these identifications we obtain 
\begin{equation}\label{eq:diaK}
z\dia \xi=\psi(z\xi),~ \text{where}\; \xi\in \bfD_{Q^\vee}\subset \imath(\cQ_{Q^\vee}),\, z\in \bfD_{W_\aff}\subset \cQ_{W_\aff}, 
\end{equation}
and  Lemma~\ref{lem:pr} gives 
\begin{equation}
\label{eq:projK} \psi(\xi z)=\xi \psi(z).\end{equation}

\begin{example}Consider the affine root system of extended Dynkin type~$\hat A_2$. It has three simple roots $\al_0,\al_1,\al_2$ and the highest root $\theta=\al_1+\al_2$. Denote $X_{ij}=X_{i}X_j$ for simplicity.  Direct computations then give:
\begin{align*}
\psi(X_{10}) &= X_1\dia X_0=\tfrac{1}{x_1}Z_{\al_1+\al_2}-\tfrac{1}{x_1}Z_{\al_2},\\
\psi(X_{20}) &=\tfrac{1}{x_2}Z_{\al_1+\al_2}-\tfrac{1}{x_2}Z_{\al_1},\\
\psi(X_{210}) &=X_2\dia \psi(X_{10}).
\end{align*}
\end{example}

Finally, we describe the centre of FADA.

\begin{lemma}\label{lem:diainv}
{\rm (i)} For any  $\xi\in \cQ_{Q^\vee}$ and $\al_i\in I$, we have 
\[
\eta_i\dia \xi=\xi~\Longleftrightarrow~\eta_i\xi=\xi \eta_i.
\]
Moreover,  if this condition is satisfied, we have
\[
c\eta_i\dia(\xi\xi')= \xi(c\eta_i\dia\xi'). 
\]
\noindent
{\rm (ii)} The centres of $\cQ_{W_\aff}$ and $\bfD_{W_\aff}$ can be described as follows:
\begin{align*}
Z(\cQ_{W_\aff}) &=\{\xi\in \cQ_{Q^\vee}\mid \eta_w\xi=\xi \eta_w, \forall w\in W\}=(\cQ_{Q^\vee})^W,\\
Z(\bfD_{W_\aff}) &=\{\xi\in \bfD_{Q^\vee}\mid\eta_w\xi=\xi \eta_w, \forall w\in W\}=(\bfD_{Q^\vee})^W.
\end{align*}
\noindent
{\rm (iii)} There are ring homomorphisms
\[
\cQ_{W_\aff}\to \End_{(\cQ_{Q^\vee})^W}(\cQ_{Q^\vee}), ~
\bfD_{W_\aff}\to \End_{(\bfD_{Q^\vee})^W}(\bfD_{Q^\vee}), ~ z\mapsto z\dia \text{-}. 
\]
\end{lemma}

\begin{proof}
(i) For a given $\xi=\sum_\la c_{\la} \eta_{t_\la}$, $c_\la\in \cQ$ we get
\[\eta_i\xi=\sum_{\la}s_i(c_{\la})\eta_i\eta_{t_\la}=\sum_\la s_i(c_{\la})\eta_{t_{s_i(\la)}}\eta_{i}=\sum_{\la'} s_i(c_{s_i(\la')})\eta_{t_{\la'}}\eta_i,
\text{ where }\la'=s_i(\la).\] 
On the other side, we have
\[
\eta_i\dia\xi=\psi(\eta_i\xi)=\psi(\eta_i\sum_\la c_\la \eta_{t_\la})=\sum_\la s_i(c_{s_i(\la)})\eta_{t_\la}.
\]
Therefore, $\eta_i\dia \xi=\xi$ if and only if $s_i(c_{s_i(\la)})=c_\la$ for any $\la\in Q^\vee$, which is equivalent to say that $\eta _i\xi=\xi\eta_i$.

Now if this condition is satisfied, then 
\[
c\eta_i\dia (\xi \xi')=\psi(c\eta_i\xi \xi')=\psi(\xi c\eta_i\xi')\overset{\eqref{eq:projK}} =\xi \psi(c\eta_i\xi')\overset{\eqref{eq:diaK}}=\xi(c \eta_i\dia \xi').
\]

(ii) Since $\cQ_{Q^\vee}=\imath(\cQ_{Q^\vee})=C_{\cQ_{W_\aff}}(\cQ)$, we have $Z(\cQ_{W_\aff})\subset \cQ_{Q^\vee}$, and the first identity then follows. By part~(i) we know that  $\eta_i\dia \xi=\xi$, $\forall \al_i\in I$ is equivalent to  $\eta_w \xi=\xi\eta_w$, $\forall w\in W$ that is equivalent to $z\xi=\xi z$, $\forall z\in \cQ_{W}$. Since $\xi$ already commutes with $\eta_{t_\la}$, $\la\in Q^\vee$, $\xi$ belongs to the centre $Z(\cQ_{W_\aff})$. Conversely, if $\xi\in Z(\cQ_{W_\aff})\cap \cQ_{Q^\vee}$, then it is invariant under all $\eta_i$, $\al_i\in I$.  

The description of the centre $Z(\bfD_{W_\aff})$ follows similarly. 

(iii) Follows from parts (i) and (ii). 
\end{proof}

%%%%%%%%%%%%%%%%%%%%%%%%%%%%%%%%%%%%%%%%%%

\section{Borel isomorphisms}

In this section we study Borel isomorphisms involving the FADA and the Peterson subalgebra. 
We assume $\bbQ \subset R$ throughout this section. 

Consider the left $S$-linear dual $\bfD_W^*$ embedded into $\cQ_W^*$. The latter has a $\cQ$-basis $f_w$, $w\in W$.
Following~\cite[\S~11]{CZZ1} there is an (equivariant) characteristic map 
\begin{equation}\label{subsec:char}
\fc\colon S\to \bfD_W^*,\; a\mapsto a\bullet \sum_{w\in W}f_w=\sum_{w\in W}w(a)f_w
\end{equation}
which induces the Borel isomorphism (see \cite[Theorem 11.4]{CZZ1})
\begin{equation}\label{subsec:borel}
\rho:S\otimes_{S^W}S\to \bfD_W^*, \quad a\otimes b\mapsto a\fc(b)=\sum_{w\in W}aw(b)f_w.
\end{equation}
Recall that 
$\sigma= \sum_{w\in W}\eta_{w}\in \bfD_{W_\aff}$. 
Denote $\bfx=\prod_{\al\in \Phi^+}x_{-\al}$ and 
$Y = \sigma \tfrac{1}{\bfx}$. 
By~\cite[Lemma 10.12]{CZZ2} ($Y=Y_\Pi$) we have $Y\in \bfD_W$. 

\begin{lemma}\label{lem:absorbD} 
We have  $\sigma  \bfD_W=Y\bfD_W=Y S$. 
\end{lemma}

\begin{proof}
Observe that $Y=\tfrac{1}{|W|} \sigma Y$, so $Y S\subset Y\bfD_{W}\subset \sigma Y \bfD_{W}\subset \sigma \bfD_W$. 
Conversely, $\bfD_W$ is also a right $S$ module with basis $X_{I_v}, v\in W$, and $Y X_{I_v}=\de_{v,e} Y$.  
So given $X_{I_v}b\in \bfD_{W}$ with $v\in W$, $b\in S$, we have 
\[
\sigma X_{I_v} b=\sigma \de_{v,e}b=\de_{v,e}Y \bfx b\in Y S. 
\]
So $\sigma \bfD_W\subset Y S$, and the result follows. 
\end{proof}

\begin{lemma}\label{lem:SpiS=DW}
We have  $SYS=\bfD_W$. So $\bfD_W$ is a cyclic $S$-$S$-bimodule. 
\end{lemma}

\begin{proof}
According to~\cite[Lemma 10.3]{CZZ2} $\bfx f_e\in \bfD_W^*$. 
Let $\sum_i a_i\otimes b_i\in S\otimes_{S^W}S$ so that $\rho(\sum_i a_i\otimes b_i)=\bfx f_e$. Then 
\[
\sum_i a_iw(b_i) = \begin{cases}
\bfx, & w=e,\\
0, & \text{otherwise}. 
\end{cases}
\]
Therefore,
\[
\sum_i a_iY b_i = \sum_{w\in W} \sum_i\tfrac{a_i w(b_i)}{\bfx}\eta_{w}=1.
\]
Finally, by Lemma~\ref{lem:absorbD} for any $z\in \bfD_W$ 
we can write $z=\sum_i a_iY b_iz=\sum_i a_iY b'_i$ for some $b'_i\in S$. 
\end{proof}

\begin{lemma} We have 
$\psi(\sigma \bfD_{W_\aff}) = Z(\bfD_{W_\aff})$. 
\end{lemma}

\begin{proof}
Let $z=c\eta_{t_{\lambda}u}$, where $c\in \cQ$, $u\in W$. We have
\begin{align*}
\psi(\sigma z) & 
=\psi\Big(
    \sum_{w\in W} \eta_w c\eta_{t_{\lambda}u}\Big)
= \psi\Big(
    \sum_{w\in W} w(c)\eta_{t_{w(\lambda)}}\eta_{wu}\Big)
= \sum_{w\in W} w(c)\eta_{t_{w(\lambda)}}.
\end{align*}
Observe that elements in the image $\imath(\cQ_{Q^\vee})\subset \cQ_{W_\aff}$ already commute with elements of $S$ 
and $\eta_{t_\la}$, $\la\in Q^\vee$, so it suffices to show $\psi(\sigma z)$ commutes with $\eta_{v}$ with $v\in W$. We have
\begin{align*}
\eta_{v}\sum_{w\in W} w(c)\eta_{t_{w\lambda}}
& 
= \sum_{w\in W} vw(c)\eta_{t_{vw(\lambda})}\eta_{v}
= \bigg(\sum_{u\in W} u(c)\eta_{t_{u(\lambda)}}\bigg)\eta_{v}. 
\end{align*}
Therefore, $\psi(\sigma z)\in Z(\bfD_{W_\aff})$, so we obtain $\psi(\sigma \bfD_{W_\aff})\subset Z(\bfD_{W_\aff})$. 

As for the opposite inclusion, take $z\in Z(\bfD_{W_a})$.   
Since $Z(\bfD_{W_\aff})\subset C_{\bfD_{W_\aff}}(S)=\bfD_{Q^\vee}$, we get $\psi(z)=z$. 
Observe that $\pr(z'\sigma)=|W|\pr(z')$ for any $z'\in \bfD_{W_\aff}$. So we obtain 
\begin{align*}
    \psi\big(\sigma z\tfrac{1}{|W|}\big)& 
    = \psi\big(z\sigma \tfrac{1}{|W|}\big) 
=\psi(z) = z.
\end{align*}
Thus, $z\in \psi(\sigma  \bfD_{W_\aff})$, and the proof is finished
\end{proof}

Consider two ring homomorphisms induced by the usual mutiplication:
\begin{align}\label{eq:thetaQ}
\Theta\colon S\otimes_{S^W} Z(\bfD_{W_\aff}) & \longrightarrow \bfD_{Q^\vee}.\\\label{eq:thetaW}
\Xi\colon \bfD_{W}\otimes_{S^W} Z(\bfD_{W_\aff}) & \longrightarrow \bfD_{W_\aff}.
\end{align}
Note that in the definition of $\Theta$ and $\Xi$ one can switch the tensor factors. 
Moreover, both homomorphisms are left $S$-linear. 
%\begin{align*}
%S\otimes_{S^W} Z(\bfD_{W_\aff})\stackrel{\sim}\longrightarrow Z(\bfD_{W_\aff})\otimes_{S^W} S& \longrightarrow \bfD_{Q^\vee}.\\ \bfD_{W}\otimes_{S^W} Z(\bfD_{W_\aff})\stackrel{\sim}\longrightarrow Z(\bfD_{W_\aff})\otimes_{S^W} \bfD_{W}&\longrightarrow \bfD_{W_\aff}.
%\end{align*}
The following is our first main result.
\begin{theorem}\label{thm:main1} Assume $\bbQ\subset R$. The maps $\Theta$ and $\Xi$ are ring isomorphisms.  
\end{theorem}

The proof will occupy the rest of this section. We start proving the surjectivity first.

\begin{prop}\label{thm:DQ=SZ}
The map $\Theta\colon S\otimes_{S^W} Z(\bfD_{W_\aff}) \to \bfD_{Q^\vee}$ given in \eqref{eq:thetaQ} is surjective. 
\end{prop}
\begin{proof}
Consider the following diagram
\[
\xymatrix{S\otimes_{S^W}\sigma \bfD_{W_\aff}\ar[r]\ar[d]_{\id \otimes \psi} & \bfD_{W_\aff}\ar@{->>}[d]^-{\psi}\\
S\otimes_{S^W}Z(\bfD_{W_\aff})\ar[r] ^-\Theta& \bfD_{Q^\vee}}. 
\]
Since $\psi$ is an $S$-module homomorphism, this diagram commutes.

By Lemma~\ref{lem:SpiS=DW} we can write $1=\sum_i a_iY b_i$ for some $a_i$, $b_i\in S$. 
For any $z\in \bfD_{W_\aff}$, we then have
$z=\sum_i a_i Y b_i z$. 
This shows that elements of $Y \bfD_{W_{\aff}}$ generate $\bfD_{W_\aff}$ as a left $S$-module. 
Similarly to the proof of~Lemma \ref{lem:absorbD} we obtain that $Y \bfD_{W_\aff}=\sigma \bfD_{W_\aff}$. 
So the top horizontal map is surjective, and the result follows. 
\end{proof}

\begin{prop}\label{prop:DQ=DZ}
The map $\;\Xi\colon \bfD_{W}\otimes_{S^W} Z(\bfD_{W_\aff}) \to \bfD_{W_\aff}$ given in \eqref{eq:thetaW} is surjective. 
\end{prop}

\begin{proof}
Since the elements of $\bfD_W$ and of $Z(\bfD_{W_\aff})$ commute with each others, the image of $\Xi$ is the subalgebra 
generated by $\bfD_W$ and $Z(\bfD_{W_{\aff}})$.  
It contains $S$ and $X_i$ for $\al_i\in I$ by definition, so  
it suffices to show that  it contains $X_0$ as well. 
Observe that 
\[
X_0  = \tfrac{1}{x_{\theta}}(1-\eta_{s_{\theta}}\eta_{t_{-\theta^\vee}}) =  \tfrac{1}{x_{\theta}}(1-\eta_{s_{\theta}})
+\eta_{s_{\theta}}\tfrac{1}{x_{-\theta}}(1-\eta_{t_{-\theta^\vee}}) =X_{\theta} + \eta_{s_{\theta}}\tfrac{x_{\theta}}{x_{-\theta}}Z_{\theta}. 
\]
Since $Z_{\theta}\in \bfD_{Q^\vee}$ by \cite[Lemma 4.1]{Z23}, we have $X_\theta\in \bfD_W$ and $\eta_{s_\theta}\in \bfD_W$. So $X_0$ belongs to the subalgebra generated by $\bfD_W$ and $Z(\bfD_{W_\aff})$. 
\end{proof}

\begin{corollary}\label{thm:DQDWDWa}
The maps $\bfD_{W}\otimes_{S}\bfD_{Q^\vee}\to \bfD_{W_\aff}$ and 
$\bfD_{Q^\vee}\otimes_{S}\bfD_{W}\to 
\bfD_{W_\aff}$ induced by the usual multiplication are  isomorphisms of left $S$-modules
(In the first map $\bfD_W$ is viewed as an  $S$-$S$-bimodule, and in the second map $\bfD_W$ is viewed as a left $S$-module.) 
\end{corollary}

\begin{proof}
Since $\bfD_{Q^\vee}\supset Z(\bfD_{W_\aff})$, by Proposition~\ref{prop:DQ=DZ} both maps are surjective. 
To prove the injectivity we change the base to $\cQ$-modules by applying the exact functors $\text{-}\otimes_S\cQ$ and $\cQ\otimes\text{-}$.  It then suffices to show that the induced maps
$$
\bfD_{W}\otimes_{S}\bfD_{Q^\vee}\otimes_S\cQ =
\cQ_W\otimes_{\cQ} \cQ_{Q^\vee}
\longrightarrow \cQ_{W_\aff}
=\bfD_{W_\aff}\otimes_{S} {\cQ}$$
$$
\cQ\otimes_S\bfD_{Q^\vee}\otimes_{S}\bfD_{W} =
\cQ_{Q^\vee}\otimes_{\cQ} \cQ_{W}
\longrightarrow \cQ_{W_\aff}
={\cQ}\otimes_{S} \bfD_{W_\aff}$$
are injective. But these are even isomorphisms. So, the conclusion follows. 
\end{proof}

We now discuss injectivity of the maps in the theorem.

Consider the $\cQ$-linear dual $\cQ^*_{W^P}=\Hom(W^P,\cQ)$ with basis $f_{w}$, $w\in W^P$. 
One can also identify it with the invariants $(\cQ_W^*)^{W_P}$ by identifying $f_w$, $w\in W^P$ with $\sum_{v\in W_P}f_{wv}\in (\cQ_W^*)^{W_P}$ (see~\cite[\S11]{CZZ2} for more details). 

\begin{lemma}\label{lem:injBorel}
For any parabolic subgroup $W_P$ of $W$ the following map is an isomorphism
\[\rho_{P,\cQ}\colon S\otimes_{S^W}\cQ^{W_P}\longrightarrow \cQ_{W^P}^*,\qquad
\rho_{P,\cQ}(c_1\otimes c_2)=\sum_{w\in W^P}c_1w(c_2)f_w.\]
\end{lemma}

\begin{proof}
Assume first that $P=B$. Then the map $\rho_{P,\cQ}$ is obtained from the isomorphism $\rho$ by the base change with the functor 
$\text{-}\otimes_S\cQ$. So $\rho_{B,\cQ}$ is an isomorphism.

For a general parabolic $W_P$ there is a commutative diagram 
$$\xymatrix{
S\otimes_{S^W} \cQ^{W_P}\ar[rr]^-{\rho_{P,\cQ}}\ar[d] & &
\cQ_{W^P}^* \ar[d]\\
S\otimes_{S^W} \cQ\ar[rr]_-\sim^-{\rho_{B,\cQ}}&&
\cQ_W^*}.
$$
Both vertical maps identify the top with the $W_P$-invariant subsets of the bottom, so the top horizontal map is an isomorphism. 
\end{proof}

\begin{lemma}\label{lem:injtheta}
The map $\Theta:S\otimes_{S^W} Z(\bfD_{W_\aff}) \to \bfD_{Q^\vee}$ defined in \eqref{eq:thetaQ} is injective. 
\end{lemma}
\begin{proof}
Let $z=\sum_{\lambda\in Q^\vee} c_{\lambda}\eta_{t_{\lambda}}\in Z(\bfD_{W_\aff})\subset \bfD_{Q^\vee}$ with $c_\la\in \cQ$. Since $\eta_u z=z\eta_u$ for any $u\in W$, we have
\begin{equation}\tag{*}
\forall u\in W,\quad uc_{\lambda} =c_{u(\lambda)}.
\end{equation}
These properties give us an injective map:
\[
\phi: Z(\bfD_{W_\aff})\to \bigoplus_{\lambda\in Q^\vee_{\geq0}}\cQ^{W_{\lambda}},\qquad  
\sum_{\lambda\in Q^\vee} c_{\lambda}\eta_{t_{\lambda}}
\longmapsto 
(c_{\lambda})_{\lambda\in Q^\vee_{\geq0}},
\]
where $Q^\vee_{\ge 0}$ is the set of dominant coroots and $W_{\lambda}$ is the stabilizer of $\lambda$, which is a parabolic subgroup of $W$. 

Let $W^\la$ denote the set of minimal length representatives of the cosets $W/W_\la$.
Consider the following diagram 
$$\xymatrix{
S\otimes_{S^W} Z(\bfD_{W_\aff})
\ar[r]^-{\Theta}\ar[d]_{\operatorname{id}\otimes \phi} & \bfD_{Q^\vee}\ar@{^(->}[d] \rule[-0.6pc]{0pc}{0pc} &\\
\bigoplus_{\lambda\in Q^\vee_{\geq0}}(
S\otimes_{S^W}\cQ^{W_{\lambda}})\ar[r]^-{\Theta'} &  
\bigoplus_{\lambda\in Q^\vee_{\geq0}}\big(
\bigoplus_{w\in W^{\lambda}} \cQ\eta_{t_{w(\lambda)}}\big)\ar@{=}[r]& \cQ_{Q^\vee},
}$$
where the map $\Theta'$ is a direct sum of maps
$$S\otimes_{S^W} \cQ^{W_{\lambda}}\longrightarrow 
\bigoplus_{w\in W^{\lambda}} \cQ\eta_{t_{w(\lambda)}},\quad 
c_1\otimes c_2\longmapsto \sum_{w\in W^{\lambda}}
c_1 w(c_2)\eta_{t_{w(\lambda)}},$$
for all $\lambda\in Q_{\geq0}^\vee$. 
Since by Lemma \ref{lem:injBorel} each such component map is injective, so is $\Theta'$. 

By direct computations and by the property (*)  the diagram is commutative. 
Since both maps $\operatorname{id}\otimes \phi$ and $\Theta'$ are injective, so is $\Theta$.
\end{proof}

\begin{prop}
The map $\Xi\colon \bfD_{W}\otimes_{S^W} Z(\bfD_{W_\aff}) \to \bfD_{W_\aff}$ defined in \eqref{eq:thetaW} is injective. 
\end{prop}

\begin{proof}
It follows from the combination of previous lemmata: 
\begin{align*}
\bfD_{W}\otimes_{S^W} Z(\bfD_{W_\aff})
& \simeq 
\bfD_{W}\otimes_{S}S\otimes_{S^W} Z(\bfD_{W_\aff})\\
& \simeq \bfD_{W}\otimes_S \bfD_{Q^\vee}
& \text{($\Theta$ is an isomorphism)}\\
& \simeq 
 \bfD_{W_\aff}
& \text{(Proposition \ref{thm:DQDWDWa})}.
& \qedhere
\end{align*}
\end{proof}

%%%%%%%%%%%%%%%%%%%%%%%%%%

\section{The $\hat A_1$ case}

In this section, we discuss an example of the formal Peterson subalgebra for an affine root system of type $\hat A_1$. 
We show that it provides a natural model for `quantum' oriented cohomology of $\bbP^1$. 

Recall that a root system of type $\hat A_1$ has two  simple roots, $\al_1=\theta=\al$ and $\al_0=-\al+\de$, and each $w\in W_\aff$ has a unique reduced decomposition. We follow the notation of~\cite[\S4.3]{LSS10} and define for $i\ge 1$:
\begin{center}
\begin{tabular}{lcr@{$\,=\,$}lcr@{$\,=\,$}l}
$\sigma_0=e$, & & $\sigma_{2i}$ & $(s_1s_0)^i=t_{-i\al^\vee}$, & &  $\sigma_{2i+1}$ & $s_0\sigma_{2i}$, \\
& & $\sigma_{-2i}$ & $(s_0s_1)^i=t_{i\al^\vee}$, & & $\sigma_{-(2i+1)}$ & $s_1\sigma_{-2i}$.
\end{tabular}
\end{center}
The set of minimal length coset representatives of $W_\aff/W$ is then $W_\aff^-=\{\sigma_i\mid i\ge 0\}$. 

The root lattice is $Q=\bbZ \al$, and in the formal group algebra $S=R[[Q]]_F$ we have $x_n=x_{n\al}=n\cdot_F x_{\al}$, where $n\cdot_F x$, $n\in \bbZ$ is the $n$-fold formal sum (inverse) of $x$. As for Demazure elements, we have $X_j^2=\kappa_{\al_j} X_j$, where $j=0,1$ and
$\ka_{\al_j}=\kappa_\al=\frac{1}{x_\al}+\frac{1}{x_{-\al}}$ is $W_\aff$-invariant. 

Set $\mu=-\frac{x_{-1}}{x_1}$. Observe that if $F$ is of the form $F(x,y)=\frac{x+y-\beta xy}{g(x,y)}$ for some power series $g(x,y)$, we have $x_{-1}=\tfrac{x_1}{\beta x_1-1}$, hence, $\mu=\tfrac{1}{1-\beta x_1}$.
Given a reduced expression $w=s_is_j\ldots$, we will use the notation $Y_{ij\cdots}$ for $Y_w=Y_{I_w}$. 
Denote $\frakX_w=\pr(X_w)$, $\frakY_w=\pr(Y_w)$.

\begin{example}\label{ex:comp}
Direct computations give:
\begin{align*}
\imath(\frakX_0)&=\imath(\frakX_{\sigma_1})=X_0+X_1-x_{-1}X_{01}     \\
\imath(\frakX_{10})&=\imath(\frakX_{\sigma _2})=X_{10}+\mu X_{01} \\
\imath(\frakX_{010})&=\imath(\frakX_{\sigma _3})=X_{010}+X_{101}-x_{-1}X_{1010}
\end{align*}
and
\begin{align*}
\frakY_0 &=\frakY_{\sigma_1}=\tfrac{1}{x_1}+\tfrac{1}{x_{-1}}\eta_{t_{\al^\vee}} \\
\frakY_{10} &=\frakY_{\sigma_2}=Y_1\dia \frakY_0=\tfrac{2}{x_1x_{-1}}+\tfrac{1}{x_{-1}^2}\eta_{t_{\al^\vee}}+\tfrac{1}{x_1^2}\eta_{t_{-\al^\vee}} \\
\frakY_{010} &=\frakY_{\sigma_3}=\frakY_0\frakY_{10}=\tfrac{3}{x_1^2x_{-1}}+\tfrac{3}{x_1x_{-1}^2}\eta_{t_{\al^\vee}}+\tfrac{1}{x_1^3}\eta_{t_{-\al^\vee}}+\tfrac{1}{x_{-1}^3}\eta_{t_{2\al^\vee}}.
\end{align*}
These computations also show that $\frakX_{\sigma_{i}}$,  $i=1,2,3$ satisfy identities similar to those of~\cite[Lemma 3]{LLMS18}.
\end{example}
We now look at various products of elements $\frakY_w\in \bfD_{Q^\vee}$. 

\begin{lemma}\label{lem:mult}
For each $i\ge 1$ and $w\in W_\aff^-$ we have $\frakY_{w\sigma_{2i}}=\frakY_{w}\frakY_{\sigma_{2i}}$.
In particular, we get $\frakY_{\sigma_{2i}}=\frakY_{\sigma_2}^i=\frakY_{10}^i$ and, therefore, $\{\frakY_{\sigma_{2i}}\mid i\ge 1\}$ is a multiplicative set.
\end{lemma}

\begin{proof}
Observe that $\sigma_{2i}=s_1s_0s_1\cdots s_0$, so $Y_{\sigma_{2i}}=Y_1Y_{\sigma_{2i-1}}=(1+\eta_1)\frac{1}{x_{-\al}}Y_{\sigma_{2i-1}}$. By Lemma \ref{lem:pr} we get
\[
\frakY_{w\sigma_{2i}}=\pr(Y_{w}Y_{\sigma_{2i}})=\pr(Y_w)\pr(Y_{\sigma_{2i}})=\frakY_w\frakY_{\sigma_{2i}}. \qedhere
\]
\end{proof}

\begin{lemma}\label{lem:dia} 
For each $i\ge 1$ we have $\frakY_{\sigma_{2i}}\in (\bfD_{Q^\vee})^W$, and for $j=0,1$
\[
Y_j\dia \frakY_{w}=
\begin{cases}
\frakY_{s_iw}, &\text{if }s_jw>w,\\
\ka_\al\frakY_{w}, &\text{ if }s_j<s_jw.
\end{cases}
\]
In particular, we have
\begin{equation}\label{eq:cyclic}
Y_{\sigma_i}\dia \frakY_0=Y_{\sigma_i}\dia \frakY_{\sigma_1}=
\begin{cases}
\frakY_{\sigma_1}, &\text{if }i=0, \\
\kappa_\al\frakY_{\sigma_i}, &\text{ if }i>0,\\
\frakY_{\sigma_{-i+1}}, &\text{ if }i<0.
\end{cases}
\end{equation}
\end{lemma}

\begin{proof}
Since $\sigma_{2i}=s_1\sigma_{2i-1}$, we get $Y_{\sigma_{2i}}=(1+\eta_1)\tfrac{1}{x_{-\al}}Y_{\sigma_{2i-1}}$, which implies that  $\eta_1Y_{\sigma_{2i}}=Y_{\sigma_{2i}}$. By Lemma~\ref{lem:pr} we then obtain
\[
\eta_1\dia \frakY_{\sigma_{2i}}=\eta_1\dia\pr(Y_{\sigma_{2i}})=\pr(\eta_1Y_{\sigma_{2i}})=\pr(Y_{\sigma_{2i}})=\frakY_{\sigma_{2i}},
\]
therefore, $\frakY_{\sigma_{2i}}\in (\bfD_{Q^\vee})^W$. Similarly, we obtain $Y_j\dia \frakY_{w}=\pr(Y_jY_{w})$, and the formula for the action follows.
\end{proof}

\begin{corollary} \label{cor:cyclic}
The set $\bfD_{Q^\vee}$ is a cyclic $\bfD_{W_\aff}$-module, generated by $\frakY_{0}=\frakY_{\sigma_1}$. 
Moreover, the kernel of the  map $\pi\colon \bfD_{W_\aff}\to \bfD_{Q^\vee}$, $z\mapsto z\dia \frakY_{\sigma_1}$ is $\bfD_{W_\aff}X_0$.
\end{corollary}

\begin{proof}The first part follows from \eqref{eq:cyclic}. For the second part we have \[X_0\dia \frakY_{\sigma_1}=(\kappa_\al-Y_0)\dia\frakY_{\sigma_1}=\kappa_\al\frakY_{\sigma_1}-\kappa_\al\frakY_{\sigma_1}=0, \] so $X_0\in \ker\pi$. 

Conversely, let $z=\sum_{i\ge 1}a_iY_{\sigma_i}+\sum_{j\ge 0}b_jY_{\sigma_{-j}}\in \ker \pi$. We have
\begin{eqnarray*}
z\dia \frakY_{\sigma_1}&=&(\sum_{i\ge 1}a_iY_{\sigma_i}+\sum_{j\ge 0}b_jY_{\sigma_{-j}})\dia \frakY_{\sigma_1}=\sum_{i\ge 1}a_i\kappa_\al \frakY_{\sigma_{i}}+\sum_{j\ge 0}b_j\frakY_{\sigma_{j+1}}\\
&=&\sum_{i\ge 1}a_i\kappa_\al \frakY_{\sigma_i}+\sum_{k\ge 1}b_{k-1}\frakY_{\sigma_k}=\sum_{i\ge 1}(a_i\kappa_\al +b_{i-1})\frakY_{\sigma_i}.
\end{eqnarray*}
Therefore, if $z\dia \frakY_{\sigma_1}=0$, then $b_{i-1}=-a_i\kappa_\al$ for all $i\ge 1$, and we obtain 
\begin{align*}
z&=\sum_{i\ge 1}a_iY_{\sigma_i}-\sum_{j\ge 0}\kappa_\al a_{j+1}Y_{\sigma_{-j}}=\sum_{i\ge 1}a_iY_{\sigma_{1-i}}Y_0-\sum_{k\ge 1}\kappa_\al a_{k}Y_{\sigma_{1-k}}\\
&=\sum_{i\ge 1}a_iY_{1-i}(Y_0-\kappa_\al)=(\sum_{i\ge 1}a_iY_{1-i})(-X_0)\in \bfD_{W_\aff}X_0.\qedhere
\end{align*}
\end{proof}

\begin{remark}
Observe that the map $\pr:\bfD_{W_\aff}\to \bfD_{W_\aff/W}\cong \bfD_{Q^\vee}$ has a kernel $\bfD_{W_\aff}X_1=\oplus_{i<0}SX_{\sigma_i}$,  while the map $\pi$ has a kernel $\bfD_{W_\aff}X_0=\oplus_{i> 0}SX_{\sigma_i}$.

For a general affine root system (for an extended Dynkin diagram), one can show that  $\bfD_{Q^\vee}$ is a cyclic left  $\bfD_{W_\aff}$-module, generated by $\pr(Y_0)$. 
\end{remark}

From the identities of Example~\ref{ex:comp} it follows that
\[
\frakY_0^2=x_{-1}\frakY_{010}+\mu\frakY_{10}, 
\]
and we obtain the following presentation of the formal Peterson algebra in terms of generators and relations:
\begin{equation}\label{eq:DQrank1}
\bfD_{Q^\vee}\cong S[\fraks,\frakt]/(\fraks^2=x_{-1}\fraks\frakt+\mu \frakt), ~\frakY_0\mapsto \fraks, \frakY_{10}\mapsto \frakt. 
\end{equation}

According to Lemma \ref{lem:mult} we may define the localization 
\[
\bfD_{Q^\vee, \loc}=\bfD_{Q^\vee}[\tfrac{1}{\frakY_{2i}}, i\ge 1]. 
\]
From \eqref{eq:DQrank1} we then obtain our second main result:
\begin{theorem}\label{thm:main2}
We have the following presentation
\[\bfD_{Q^\vee, \loc}\cong S[\frakt,{\frakt}^{-1}][\fraks]/(\fraks^2=x_{-1}\fraks\frakt+\mu \frakt).\]
\end{theorem}
In particular, the action of $\bfD_{W_\aff}$ on $\bfD_{Q^\vee}$ extends to an action on the localization $\bfD_{Q^\vee, \loc}$ by 
\[
z\dia (\frac{\xi}{\frakY_{2i}})=\frac{z\dia \xi}{z\dia \frakY_{2i}},\quad i\ge 1,\, \xi\in \bfD_{Q^\vee}.
\]
Observe that it is well-defined since for any $i,j\ge 1$ we have
\[
z\dia (\frac{\xi \frakY_{\sigma_{2j}}}{\frakY_{\sigma_{2{(i+j)}}}})=\frac{z\dia (\xi\frakY_{\sigma_{2j}})}{z\dia (\frakY_{\sigma_{2i}}\frakY_{\sigma_{2j}})}\overset{\ref{lem:dia}}=\frac{(z\dia \xi)\frakY_{\sigma_{2j}}}{(z\dia \frakY_{\sigma_{2i}})\frakY_{\sigma_{2j}}}=\frac{z\dia\xi}{z\dia \frakY_{\sigma_{2i}}}=z\dia \frac{\xi}{\frakY_{\sigma_{2i}}}. 
\]
It then follows from Corollary \ref{cor:cyclic} that 
\begin{corollary}
The localized algebra $\bfD_{Q^\vee, \loc}$ is a cyclic $\bfD_{W_\aff}$-module generated by $\frakY_0$. 
\end{corollary}

\begin{remark} Let $F=x+y-\be x y$. Observe that for cohomology ($\beta=0$) and $K$-theory ($\beta=1$) the localization $\bfD_{Q^\vee, \loc}$ computes quantum cohomology and quantum $K$-theory of $\bbP^1$ respectively. For instance, for $K$-theory the presentation \eqref{eq:DQrank1} recovers that of \cite[(17)]{LLMS18}. 
Therefore, it makes sense to think of $\bfD_{Q^\vee, \loc}$ as a model for `quantum' oriented cohomology of the projective line $\bbP^1$.
\end{remark}

Finally, recall from Proposition \ref{prop:hopf} that $\bfD_{Q^\vee}$ is a Hopf algebra with coproduct defined by
\[
\triangle(a\eta_{t_\la})=a\eta_{t_\la}\otimes \eta_{t_\la}=\eta_{t_\la}\otimes a\eta_{t_\la}.
\]
In our case we then obtain
\begin{align*}
\triangle (\frakY_0)&=\tfrac{1}{x_1}(1-\mu)+\mu(\frakY_0\otimes 1+1\otimes \frakY_0)+x_{-1}\frakY_0\otimes \frakY_0,\\
\triangle (\frakY_{10})&=\ka_\al^2+(\tfrac{x_1}{x^2_{-1}}-\tfrac{1}{x_1})(1\otimes \frakY_0+\frakY_0\otimes 1)+(1+\tfrac{x_1^2}{x_{-1}^2})\frakY_0\otimes \frakY_0\\
&+\tfrac{1}{\mu}(\frakY_{10}\otimes 1+1\otimes \frakY_{10})-\tfrac{x_1^2}{x_{-1}}(\frakY_0\otimes \frakY_{10}+\frakY_{10}\otimes \frakY_0)+x_1^2\frakY_{10}\otimes \frakY_{10}.
\end{align*}
\begin{example}
In particular, for the cohomology we get 
\begin{align*}
\triangle (\frakY_0)&= 1\otimes \frakY_0+\frakY_0\otimes 1-x_\al\frakY_0\otimes \frakY_0,\\
\triangle (\frakY_{10})&= 2\frakY_0\otimes \frakY_0+(1\otimes \frakY_{10}+\frakY_{10}\otimes 1)+x_\al (\frakY_0\otimes \frakY_{10}+\frakY_{10}\otimes \frakY_0)+x_\al^2 \frakY_{10}\otimes \frakY_{10},
\end{align*}
and for the $K$-theory (identifying $x_\al=1-e^{-\al}$) we get
\begin{align*}
\triangle (\frakY_0) &= -e^\al+e^\al(\frakY_0\otimes 1+1\otimes \frakY_0)+(1-e^\al)\frakY_0\otimes \frakY_0,\\
\triangle (\frakY_{10}) &= 1-(1+e^\al)(\frakY_0\otimes 1+1\otimes \frakY_0)+(1+e^{-2\al})\frakY_0\otimes \frakY_0
+ e^{-\al}(\frakY_{10}\otimes 1+1\otimes \frakY_{10}) \\
&+ (e^{-\al}-e^{-2\al})(\frakY_0\otimes \frakY_{10}+\frakY_{10}\otimes \frakY_0)+(1-e^{-\al})^2\frakY_{10}\otimes \frakY_{10}.
\end{align*}
\end{example}

%%%%%%%%%%%%%%%%%%%%%%%%%%%%%%%%%

\section{The dual of the formal Peterson subalgebra}

In this section, we study the dual of the formal Peterson subalgebra. 

Consider the $\cQ$-linear dual of the twisted group algebra $\cQ_{W_\aff}^*=\Hom_{\cQ}(\cQ_{W_\aff}, \cQ)$. It is generated by $f_w$, $w\in W$.
Following~\cite{LZZ20} there are two actions of $\cQ_{W_\aff}$ on $\cQ_{W_\aff}^*$ defined as follows:
\[
a\eta_w\bullet bf_v=bvw^{-1}(a)f_{vw^{-1}}, \quad a\eta_w \odot bf_v=aw(b)f_{wv}.
\]
Indeed, the $\odot$ action comes from left multiplication in $\cQ_{W_\aff}$, and the $\bullet$ action comes from right multiplication. 
Observe that these two actions commute, which makes $\cQ_{W_\aff}^*$ into a $\cQ$-$\cQ$-bimodule.  
Moreover, $(z\bullet  f)(z')=f(zz')$, $z,z'\in \cQ_{W_\aff}$, $f\in \cQ_{W_\aff}^*$. 

We now define two tensor products. 
The first one is the tensor product $\cQ_{W_\aff}\otimes \cQ_{W_\aff}$ of left $\cQ$-modules that is 
$az\otimes z'=z\otimes az'$, $a\in \cQ$, $z,z'\in \cQ_{W_\aff}$. There is a canonical map $\Delta:\cQ_{W_\aff}\to \cQ_{W_\aff}\otimes\cQ_{W_\aff}$ given by $a\eta_w\mapsto a\eta_w\otimes \eta_w$. This map defines a co-commutative coproduct structure on $\cQ_{W_\aff}$ with co-unit $\cQ\to \cQ_{W_\aff}$, $a\mapsto a\eta_e$. 
The second tensor product $\hat\otimes$ was introduced in \cite{LXZ23}. Here we provide a different but equivalent definition: 
\[
\cQ_{W_\aff}\hat\otimes \cQ_{W_\aff}:=\cQ_{W_\aff}\times \cQ_{W_\aff}/\lag (a\eta_w, b\eta_v)-(aw(b)\eta_w , \eta_v)\rag.
\]
Observe that $\cQ_{W_\aff}\hat\otimes \cQ_{W_\aff}$ is also a left $\cQ$-module.

Similarly, we define:
\[
\cQ_{W_\aff}^*\hat\otimes \cQ^*_{W_\aff}=\cQ^*_{W_\aff}\times \cQ_{W_\aff}^*/\lag (af_w,bf_v)-(aw(b)f_w,f_v)\rag.
\]
By definition, there is an isomorphism of $\cQ$-modules:
\[
 \cQ_{W_\aff}^*\hat\otimes \cQ_{W_\aff}^*\cong (\cQ_{W_\aff}\hat\otimes \cQ_{W_\aff})^*, \quad 
(af_w\hat\otimes bf_v)(c\eta_x\hat\otimes d\eta_y):=acw(bd)\de_{w,x}\de_{v,y}. 
\]

There is a left $\cQ$-module homomorphism defined by the product structure of $\cQ_{W_\aff}$:
\[
m\colon \cQ_{W_\aff}\hat\otimes\cQ_{W_\aff}\to \cQ_{W_\aff},\; z_1\hat\otimes z_2\mapsto z_1z_2,
\]
whose dual is given by 
\[
m^*:\cQ_{W_\aff}^*\to (\cQ_{W_\aff}\hat\otimes \cQ_{W_\aff})^*\cong \cQ_{W_\aff}^*\hat\otimes \cQ_{W_\aff}^*, \quad 
m^*(cf_w)=\prod_{u}cf_u\hat\otimes f_{u^{-1}w}. 
\]
Indeed, given any element $a\eta_u\hat\otimes b\eta_v=au(b)\eta_u\hat\otimes \eta_v$, we have
\begin{align*}
m^*(cf_w)(a\eta_u\hat\otimes b\eta_v)&=cf_w(au(b)\eta_u\eta_v)=cf_w(au(b)\eta_{uv})\\
& =\eta_{uv,w}cau(b)=(\prod_{u}cf_u\hatot f_{u^{-1}w})(a\eta_u\hatot b\eta_v). 
\end{align*}

Recall (see also \cite[\S1.7]{Z23}) that there is the Borel map defined via the characteristic map 
\[
\rho\colon \cQ\otimes_{\cQ^{W_\aff}}\cQ\to \cQ_{W_\aff}^*,\; a\otimes b\mapsto a \bfc(b)=\prod_{w\in W_\aff}aw(b)f_w.
\]
Similar to \cite{LXZ23} one obtains the following commutative diagram:
\[
\xymatrix{\cQ\otimes_{\cQ^{W_\aff}}\cQ\ar[rrr]^-{a\otimes b\mapsto a\otimes 1\hat\otimes 1\otimes b}\ar[d]^{\rho} &&&(\cQ\otimes_{\cQ^{W_\aff}}\cQ)\hat\otimes(\cQ\otimes_{\cQ^{W_\aff}}\cQ)\ar[d]^{\rho\hat\otimes \rho}\\
\cQ_{W_\aff}^*\ar[rrr] &&&\cQ^*_{W_\aff}\hatot\cQ_{W_\aff}^*}.
\]
\begin{definition}
We define the $0$-th Hochschild homology of the bimodule $\cQ_{W_\aff}^*$ to be the quotient
\[
\mathop{\rm HH_0}(\cQ_{W_\aff}^*):=\cQ_{W_\aff}^*/\lag a\bullet f-a\odot f,\; a\in \cQ, f\in \cQ_{W_\aff}\rag.
\]
\end{definition}

Consider the dual of the map $\imath\colon \cQ_{Q^\vee}\to \cQ_{W_\aff}$, $\eta_{t_\la}\mapsto \eta_{t_\la} $. It gives a surjection
\[
\imath^*\colon \cQ_{W_\aff}^*\to \cQ_{Q^\vee}^*,\; af_{t_\la w}\mapsto a\de_{w,e}f_{t_\la}.
\]
We have
\[
\imath^*(a\bullet bf_{t_\la v}-a\odot bf_{t_\la v})=\imath^*(v(a)bf_{t_\la v}-abf_{t_{\la v}})=(v(a)b-ab)\de_{v,e}f_{t_\la v}=0.
\]
So it induces a surjection 
\[
\imath^*\colon \mathop{\rm HH_0}(\cQ_{W_\aff}^*)\to \cQ_{Q^\vee}^*.
\]

On the other hand, $\ker \imath^*=\prod_{w\neq e, \la\in Q^\vee}\cQ f_{t_\la w}$. Now for any $w$, let $x_\mu \in S$ so that $w(\mu)\neq \mu$, then we obtain 
\[
f_{t_\la w}=\frac{1}{x_\mu-w(x_\mu)}(x_\mu\odot f_{t_\la w}-x_\mu \bullet f_{t_\la w})\in \ker \imath^*.
\]
Therefore, we have proven the following lemma:

\begin{lemma}\label{lem:isop}
There is an isomorphism $\mathop{\rm HH_0}(\cQ_{W_\aff}^*)\simeq \cQ_{Q^\vee}^*$ which fits into a commutative diagram
\[
\xymatrix{\cQ_{W_\aff}^*\ar@{->>}[rr]\ar@{->>}[drr]_-{\imath^*} && \HH_0(\cQ_{W_\aff}^*)\ar[d]^-\simeq \\
&& \cQ_{Q^\vee}^*}. 
\]
\end{lemma}

By definition, we have the following commutative diagram of left $S$-modules: 
\[
\xymatrix{\cQ_{W_\aff}^*\ar@{->>}[r]^-{\imath^*} & \cQ_{Q^\vee}^*\ar@{^(->}[r]^-{\pr^*} & \cQ_{W_\aff}\\
\bfD_{W_\aff}^*\ar@{->>}[r]^-{\imath'^{*}}\ar@{^(->}[u]& \bfD_{Q^\vee}^*\ar@{^(->}[u]\ar@{^(->}[r]^-{\pr^*} & \bfD_{W_\aff}\ar@{^(->}[u]}.
\]
It is also easy to see that the surjection $\imath^*\colon\cQ_{W_\aff}^*\to \cQ_{Q^\vee}^*$ induces a surjective map
\[
\imath^{*\prime}\colon \mathop{\rm HH_0}(\bfD_{W_\aff}^*)\to \bfD_{Q^\vee}^*.
\]
Our goal is to show that the isomorphism and the diagram of Lemma~\ref{lem:isop} can be restricted to the formal Peterson subalgebra $\bfD_{Q^\vee}$. Namely, we want to prove the following

\begin{theorem}\label{thm:main3} 
The map $\imath^{*\prime}$ gives an isomorphism of Hopf algebras $\mathop{\rm HH_0}(\bfD_{W_\aff}^*)\simeq \bfD_{Q^\vee}^*$.
\end{theorem}

Since the product structure on both the domain and the codomain is induced by the coproduct structure
\[
\cQ_{W_\aff}\to \cQ_{W_\aff}\otimes\cQ_{W_\aff}, \quad a\eta_w\mapsto a\eta_w\otimes \eta_w,
\]
where the codomain is the tensor product of left $\cQ$-modules, the map $\imath^{*\prime}$ is a ring homomorphism. Moreover, since the coproduct structure on both the domain and the codomain is induced by the product structure in $\cQ_{W_\aff}$ and $\cQ_{Q^\vee}$, the map $\imath^{*\prime}$ is a coalgebra homomorphism. Therefore, it only suffices to prove the injectivity of $\imath^{*\prime}$. 

To prove the latter we introduce the following filtration on the dual  $\bfD_{W_\aff}^*$ of the FADA:
\begin{definition}
Let $w_\la$ be as defined in the appendix. Set $F_i=\bigcup_{\ell(w_{\lambda})\geq i} t_\lambda W$.  
For any $f\in \cQ_{W_\aff}^*=\Hom(W_\aff, \cQ)$ set $\supp f=\{w\in W_\aff\mid f(w)\neq 0\}$. 
Define the $i$th stratum of the filtration to be
\[
\mathcal{Z}_i=\{f\in \bfD_{W_\aff}^*\mid \supp f\subseteq F_i\}.
\]
\end{definition}

We have
$\mathcal{Z}_i=\prod_{w\in F_i} S\cdot Y_{I_{w}}^*$.
Since $x\bullet af_u=u(x)af_u$ and $x\odot af_u=xa f_u$ where $x,a\in S$, $u\in W_\aff$, so each
$\mathcal{Z}_i$ is a $S$-bimodule.

Consider the following  two conditions:
\begin{gather}
\label{eq:GKMtranl}
f(\eta_{t_{\lambda}u})\in x_\alpha^{\ell_{\alpha}(w_\lambda)} S
\quad \text{ for any $u\in W$}, \\
\label{eq:GKMfin}
f(\eta_{t_{\lambda}u}-\eta_{t_{\lambda}s_{\alpha}u})
\in x_\alpha^{\ell_{\alpha}(w_\lambda)+1} S
\quad \text{ for any $u\in W$ and root $\alpha\in \Phi$}.
\end{gather}

\begin{lemma}Elements of $\calZ_i$ 
satisfy the conditions 
\eqref{eq:GKMtranl} and 
\eqref{eq:GKMfin}. 
\label{lem:newGKM}
\end{lemma}
\begin{proof}
%Note that $Z_{\pm \alpha}=\frac{1}{x_{\mp \al}}(1-\eta_{t_{\pm \alpha^\vee}})\in \bfD_{W_a}$. 
%For $d\ge 1$, $w\in W_\aff$, applying $Z_{\pm \alpha}^d\odot \_$, we get 
%$$(Z_{\pm\al}\odot f)(\eta_w)=\left(\frac{1}{x^d_{\mp \al}}(1-\eta_{t_{\pm d \al^\vee}})\odot f\right)(\eta_w)=\frac{1}{x_{\mp \al}^d}f\big((1-\eta_{t_{\pm d \alpha^\vee}})^d \eta_{w}\big)\in S.$$
%Similarly, by applying $Z_{\pm \alpha}^d\eta_{t_{\lambda}}X_\al$, we get 
%$$(Z^d_{\pm \al}\eta_{t_\la}X_\al\odot f)(\eta_u)=\frac{1}{x_\al^{d+1}}f\big((1-\eta_{t_{\pm \alpha^\vee}})^d \eta_{t_{\lambda}}(1-\eta_{s_{\alpha}}) \eta_{u}\big)\in x_{\alpha}^dS$$
%for any $d\geq 0$ and $u\in W$. 

Let $f\in \calZ_i$ and   $\ell(w_\la)=i$. 
Assume  $\left<\lambda,\alpha^\vee\right>\leq 0$. 
For $k\in [1,  \ell_{\alpha}(w_{\lambda})]$, by Lemma \ref{lem:tranwlambda} below, 
$w_{\lambda+k\al^\vee}<w_{\lambda}$ and in particular
$\ell(w_{\lambda+k\al^\vee})<\ell(w_{\lambda})$. So $f(\eta_{\la+k\al^\vee})=0$, and  we have for any $u\in W$, 
\begin{align*}
(Z_\al^{\ell_\al(w_\la)}\eta_{t_\la}\odot f)(\eta_u)&=\frac{1}{x_{-\al}^{\ell_\al(w_\la)}}f\big((1-\eta_{t_\alpha^\vee})^{\ell_{\alpha}(w_{\lambda})}\eta_{t_{\lambda}u}\big) \\
&=\frac{1}{x_{-\al}^{\ell_\al(w_\la)}}f(\eta_{t_{\lambda}u})\in S,\; \text{ and}\\
(Z_\al^{\ell_\al(w_\la)}\eta_{t_\la}X_\al\odot f)(\eta_u)&=\frac{1}{x_{-\al}^{\ell_\al(w_\la)}x_{\al}}f\big((1-\eta_{t_\alpha^\vee})^{\ell_{\alpha}(w_{\lambda})}(\eta_{t_{\lambda}u}-\eta_{t_{\lambda}s_{\alpha}u})\big) \\
&=\frac{1}{x_{-\al}^{\ell_\al(w_\la)}x_{\al}}f(\eta_{t_{\lambda}u}-\eta_{t_{\lambda}s_{\alpha}u})
\in S.
\end{align*}
Note that $\frac{x_\al}{x_{-\al}}$ is invertible in $S$, so we can replace $x_\al$ by $x_{-\al}$ whenever needed.

Similarly, if $\left<\lambda,\alpha^\vee\right><0$, then
$\ell(w_{\lambda-k\delta})<\ell(w_{\lambda})$ for $k\in [1,\ell_{\alpha}(w_{\lambda})]$. Thus 
\begin{align*}(Z_{-\al}^{\ell_{\al}(w_\la)}\odot f)(\eta_u)&=\frac{1}{x_\al^{\ell_\al(w_\la)}}f\big((1-\eta_{t_{-\alpha^\vee}})^{\ell_{\alpha}(w_{\lambda})}\eta_{t_{\lambda}u}\big) \\
&=\frac{1}{x_{\al}^{\ell_\al(w_\la)}}f(\eta_{t_{\lambda}w})\in S,\;\text{ and}\\
(Z_{-\al}^{\ell_\al(w_\la)}\eta_{t_\la}X_{\al}\odot f)(\eta_u)&=\frac{1}{x_\al^{\ell_\al(w_\la)+1}}f\big((1-\eta_{t_{-\alpha^\vee}})^{\ell_{\alpha}(w_{\lambda})}(\eta_{t_{\lambda}w}-\eta_{t_{\lambda}s_{\alpha}u})\big) \\
&=\frac{1}{x_\al^{\ell_\al(w_\la)+1}}f(\eta_{t_{\lambda}u}-\eta_{t_{\lambda}s_{\alpha}u}) 
\in S. \qedhere
\end{align*}
\end{proof}

Define 
$$\mathcal{Y}_{(i)}=\left\{f\in \Hom(F_i\backslash F_{i+1},S)\mid
f 
\text{ satisfies the conditions 
\eqref{eq:GKMtranl}}
\text{ and 
\eqref{eq:GKMfin}}
\right\}.$$
Here $f(\eta_{v_1}-\eta_{v_2}):=f(v_1)-f(v_2)$.  
Note that $\mathcal{Y}_{(i)}$ is a $S$-bimodule in the usual sense, that is $(a\odot f)(v)=af(v)$ and $(a\bullet f)(v)=v(a)f(v)$. 
By Lemma \ref{lem:newGKM}, we have a natural  projection $\mathcal{Z}_i\to \mathcal{Y}_{(i)}$, which induces an injective $S$-bimodule map 
$$\res\colon \mathcal{Z}_i/\mathcal{Z}_{i+1}\longrightarrow \mathcal{Y}_{(i)}.$$

\begin{lemma}\label{lem:Zfil}
The map $\res$ is an isomorphism of $S$-bimodules. In particular, $\calZ_i/\calZ_{i+1}$ is free of rank $|F_i\backslash F_{i+1}|$. 
\end{lemma}

\begin{proof}We only need to prove that $\res$ is surjective. 
Let $f\in \mathcal{Y}_{(i)}$. We pick a minimal element $w\in \supp(f)$. 
We first show that 
$$f(\eta_{w})\in \prod_{\alpha\in \Phi^+} x_{\alpha}^{\ell_{\alpha}(w)}S.$$
As $x_{\alpha}$'s are relative prime, it reduces to show 
\begin{equation}\label{eq:consmallGKM}
f(\eta_w)\in x_{\alpha}^{\ell_{\alpha}(w)}S
\end{equation}
for each root $\alpha$. 

If $w=t_{\lambda}u$ for $\lambda\in Q^\vee$ and $u\in W$ and $\ell_\al(w_\la)=i$. 
By Lemma \ref{lem:lengthalpha} below,  $\ell_{\alpha}(w)\in \{\ell_{\alpha}(w_{\la}),\ell_{\alpha}(w_{\la})+1\}$. 
If $\ell_{\alpha}(w)=\ell_{\alpha}(w_{\lambda})$, then \eqref{eq:consmallGKM} follows from \eqref{eq:GKMtranl} directly. 
If $\ell_{\alpha}(w)=\ell_{\alpha}(w_{\lambda})+1$, by \eqref{eq:1} and \eqref{eq:2}, we have
$$\ell_{\alpha}(t_{\lambda}s_{\alpha}u)=\ell_{\alpha}(w_{\lambda})<\ell_{\alpha}(w).$$
It implies $t_{\lambda}s_{\alpha}u<w$ (note that $t_{\lambda}s_{\alpha}u$ and $w$ are always comparable under the  Bruhat order). 
By \eqref{eq:GKMfin}, we have 
$$f(\eta_{t_{\lambda}u}-\eta_{t_{\lambda}s_{\alpha}u})
=f(\eta_{w})
\in x_\alpha^{\ell_{\alpha}(w_\lambda)+1} S.
$$

Note that the images of $Y_{I_w}^*$ for $w\in Z_i\setminus Z_{i+1}$ form a basis of $\mathcal{Z}_i/\mathcal{Z}_{i+1}$, and $Y_{I_{w}}^*(\eta_{w})=\prod_{\alpha>0} x_{\alpha}^{\ell_{\alpha}(w)}$. 
The conclusion then follows by replacing $f$ by
 $f-\frac{f(\eta_w)}{\prod_{\alpha\in \Phi^+}x_\al^{\ell_\al(w_\la)}}Y_{I_w}^*$. 
\end{proof}

Denote for each $\lambda\in Q^\vee$, $\Delta_{\lambda}= \prod_{\alpha>0} x_{\alpha}^{\ell_{\alpha}(w_{\lambda})}\in S$.
It is clear that we have an $S$-bimodule isomorphism
$$\mathcal{Y}_{(i)}\cong \bigoplus_{\ell(w_{\lambda})=i}
\Delta_\lambda \cdot \bfD_W.$$

To finish the proof of Theorem \ref{thm:main3}, we define a filtration on $
\bfD_{Q^\vee}^*$ by 
\[
\calX_i=\{f\in \bfD_{Q^\vee}^*|\supp(f)\subset F_i\}.
\]
Then $Y_{I_w}^*$ with $\ell(w_\la)=i$ is a $S$-basis of $\calX_i/\calX_{i+1}$. So the rank of $\calX_i/\calX_{i+1}$ is $|F_i\backslash F_{i+1}|$. Moreover, by definition, we know that $\imath^{*\prime}$ induces a map on each associated graded piece:
\[
\imath^{*\prime}:\calZ_i/\calZ_{i+1}\to \calX_{i}/\calX_{i+1}.
\]
From Lemma \ref{lem:Zfil},  the rank of $\calZ_i/\calZ_{i+1}$ is $|F_i\backslash F_{i+1}|$, therefore, $\imath^{*\prime}$ is an isomorphism.

%%%%%%%%%%%%%%%%%%%%%%%%%%%%%%

\appendix
\section{}

In this section, we prove several combinatorial properties of the affine Weyl group that are used in the proof of Theorem \ref{thm:main3}.

For $w\in W_a$, denote
$$\ell_{\alpha}(w)=|\{\beta=\pm \alpha+k\delta>0| w^{-1}(\beta)<0\}|.$$
It is clear that $\ell(w)=\sum_{\alpha>0} \ell_\alpha(w)$. 
Also denote by $w_{\lambda}\in W_a^-$ the minimal representative of $t_{\lambda}W$. Then 
$w_{\lambda}\leq w_{\mu}$ if and only if there exists $w\in t_{\lambda}W$ and $y\in t_{\mu}W$ such that $w\leq y$. Note that  $w\leq y$ implies $\ell_{\alpha}(w)\leq \ell_{\alpha}(y)$ for all $\al\in \Phi_+$, so after fixing $\al$ $\ell_\al(w_\la)$ becomes minimal for elements $w$ from $t_\la W$. 

\begin{lemma}\label{lem:lengthalpha}
We have the following property: 
$$\ell_{\alpha}(w_{\lambda})=
\begin{cases}
-\left<\lambda,\alpha\right>, & \text{ if }\left<\lambda,\alpha\right>\leq 0,\\
\left<\lambda,\alpha\right>-1, & \text{ if } \left<\lambda,\alpha\right>> 0.
\end{cases}$$
\end{lemma}

\begin{proof}
For $w=t_{\lambda}u\in t_{\lambda}W$, $\be=\pm \al+k\de>0$ (so $k\ge 0$), we have
\begin{equation}\label{eq:w}
w^{-1}(\be)=w^{-1}(\pm \alpha+k\delta) = \pm u^{-1}\alpha+(k\pm \left<\lambda,\alpha\right>)\delta.
\end{equation}

If $\lag\la,\al\rag\le 0$, then $w^{-1}(\be)<0$ implies $\be=\al+k\de $, and moreover,  $k\in [0,\ell_\al(w)-1]$ (since $\ell_\al(w)=|Inv_\al(w)|$). So we have
\begin{equation}\label{eq:1}\ell_{\alpha}(w)=-\left<\lambda,\alpha\right>+
\begin{cases}
1, & u^{-1}(\alpha)<0,\\
0, & u^{-1}(\alpha)>0,
\end{cases}
\end{equation}
and the minimal value is $-\lag \la,\al\rag$.

If $\left<\lambda,\alpha\right>> 0$, then $w^{-1}(\be)<0$ if and only if  $\be=-\al+k\de$ and  $k\in [1,\ell_{\alpha}(w)]$, in which case we have
\begin{equation}\ell_{\alpha}(w)=\left<\lambda,\alpha\right>-
\begin{cases}
1, & u^{-1}(\alpha)<0,\\
0, & u^{-1}(\alpha)>0.
\end{cases}
\label{eq:2}
\end{equation}
The minimal value is $\lag \la,\al\rag-1$. 
\end{proof}

\begin{lemma}\label{lem:tranwlambda}
Let   $\alpha\in \Phi^+$, $\lambda\in Q^\vee$, and $k\in [0,\ell_\al(w_\la)]$. 

If $\lag \la, \al\rag\le 0$, then 
$w_\la>w_{\lambda+k\alpha^\vee}$. If $\lag\la,\al\rag>0$, then 
$w_\la>w_{\lambda-k\alpha^\vee}$. 
\end{lemma}

\begin{proof}
If $\left<\lambda,\alpha\right>\leq 0$,  consider $w\in t_\la W$ such that $w=t_{\lambda}u$ with $u^{-1}(\alpha)<0$.  
We have \[w^{-1}(\al)=u^{-1}(\al)+\lag\la,\al\rag\de<0,\] so $w>s_\al w$. From \eqref{eq:1}, we get $\ell_{\alpha}(w)=\ell_{\alpha}(w_{\lambda})+1=-\left<\lambda,\alpha\right>+1$. 
Since $1\leq k\leq \ell_{\alpha}(w_{\lambda})=-\left<\lambda,\alpha\right>$, we get \[(s_{\alpha}w)^{-1}(-\alpha+k\delta)=u^{-1}\alpha+(k+\left<\lambda,\alpha\right>)\delta <0,\] which implies 
\[s_\alpha w>s_{-\alpha+k\delta}s_{\alpha}w=t_{k\alpha^\vee}w.\]
Therefore, $w>t_{k\al^\vee }w$. Since $w\in t_{\lambda}W$ and $t_{k\alpha^\vee}w\in t_{\lambda+k\alpha^\vee}W$, we get 
$w_{\lambda}>w_{\lambda+k\al^\vee}$. 

If $\left<\lambda,\alpha\right>> 0$, consider $w=t_{\lambda}u$ with $u^{-1}(\alpha)>0$. We have \[w^{-1}(-\alpha+\delta)
=u^{-1}(\alpha)+(1-\left<\lambda,\alpha\right>)\delta <0,\] so $w>s_{-\al+\de}w$. From \eqref{eq:2}, we have $\ell_{\alpha}(w)=\lag\la,\al\rag=\ell_{\alpha}(w_{\lambda})+1. $
Since $1\leq k\leq \ell_{\alpha}(w)-1=\ell_{\alpha}(w_{\lambda})=\left<\lambda,\alpha\right>-1$, we get
$$s_{-\alpha+\delta} w>s_{\alpha+(k-1)\delta}s_{-\alpha+\delta}x=t_{-k\alpha^\vee}w.$$
So $w>t_{-k\al^\vee}w$. Since $w\in t_{\lambda}W$ and $t_{k\alpha^\vee}w\in t_{\lambda+k\alpha^\vee}W$, we have 
$w_{\lambda}>w_{\lambda+k\alpha^\vee}$. 
\end{proof}

The following lemma is not used in this paper, but it is interesting on its own.
\begin{lemma}If $\lag\la,\al\rag=0$,  then we have the following sequence 
\[
w_{\la}<w_{\la+\al^\vee}<w_{\la-\al^\vee}<w_{\la+2\al^\vee}<w_{\la-2\al^\vee}<\cdots.
\]
If $\lag\la,\al\rag=1$, then we have the following sequence:
\[
w_{\la}<w_{\la-\al^\vee}<w_{\la+\al^\vee}<w_{\la-2\al^\vee}<w_{\la+2\al^\vee}<\cdots.
\]
The lengths $\ell_\al$ are given by $(0,1,2,3,4,\ldots)$.
\end{lemma}
\begin{proof} We only prove the case when $\lag\la,\al\rag=0$. Consider $\la+k\al^\vee$ and $\la-k\al^\vee$ with $k> 0$, then $\lag\la-k\al^\vee,\al\rag=-2k<0$, and $2k=\ell_\al(w_{\la-k\al^\vee})$, so by Lemma \ref{lem:tranwlambda}, 
\[
w_{\la-k\al^\vee}>w_{\la-k\al^\vee+2k\al^\vee}=w_{\la+k\al^\vee}.
\]

Finally, consider ${\la-k\al^\vee}$ and $\la+(k+1)\al^\vee$, $k\ge 0$, then $\lag\la+(k+1)\al^\vee,\al\rag=2(k+1)\ge 2$, and $\ell_\al(w_{\la+(k+1)\al^\vee})=2k+1$, so by Lemma \ref{lem:tranwlambda},
\[
w_{\la+(k+1)\al^\vee}\ge w_{\la+(k+1)\al^\vee-(2k+1)\al^\vee}=w_{\la-k\al^\vee}. 
\qedhere\]
\end{proof}

\newcommand{\arxiv}[1]
{\texttt{\href{http://arxiv.org/abs/#1}{arXiv:#1}}}
\newcommand{\doi}[1]
{\texttt{\href{http://dx.doi.org/#1}{doi:#1}}}
\renewcommand{\MR}[1]
{\href{http://www.ams.org/mathscinet-getitem?mr=#1}{MR#1}}

%%%%%%%%%%%%%%%%%%%%%%%%%%%%%%%%%%%

\end{document}